\documentclass{amsart}
\usepackage{amssymb}

\makeatletter
\@namedef{subjclassname@2020}{2020 Mathematics subject classification}
\makeatother

\newtheorem{thm}{Theorem}[section]
\newtheorem{cor}[thm]{Corollary}
\newtheorem{lem}[thm]{Lemma}

\newtheorem{conj}[thm]{Conjecture}
\theoremstyle{definition}

\theoremstyle{remark}

\numberwithin{equation}{section}
\begin{document}

\title[]{Distribution of residues of an algebraic
number modulo ideals of degree one}%
\author{Chunlin Wang}%
\address{School of Mathematical Siences, Sichuan Normal University, Chendu, China}%
\email{c-l.wang@outlook.com}%

\thanks{}%


\begin{abstract}
  Let $f(x)$ be an irreducible polynomial with integer coefficients of degree at least two. Hooley proved that the roots of the congruence equation $f(x)\equiv 0\mod n$ is uniformly distributed. as a parallel of Hooley's theorem under ideal theoretical setting, we prove the uniformity of the distribution of residues of an algebraic number modulo degree one ideals. Then using this result we show that the roots of a system of polynomial congruences are uniformly distributed. Finally, the distribution of digits of n-adic expansions of an algebraic number is discussed.
\end{abstract}
\maketitle

\section{Introduction}
Let $f(x)$ be a primitive, irreducible polynomial of degree $\ge 2$ with integer coefficients. The congruence equation $f(x)\equiv0\mod n$ has long been studied. One important question is how the roots $v$ of this  congruence equation are distributed as $n$ runs through all positive integers. Hooley \cite{Ho} showed that all the ratios $v/n$ with $0\le v<n, f(v)\equiv0\mod n$, when arranged as a sequence such that the dominators $n$ are ascending, is uniformly distributed. Here a sequence $x_n\in[0,1]$ is said to be uniformly distributed (with respect to the Lebesgue measure) if
$$\lim_{N\rightarrow\infty}\frac{|\{n<N:
a\le x_n<b\}|}{N}=b-a$$
for all real numbers $a$ and $b$ with $0\le a<b\le 1$, or equivalently, by Weyl's criterion \cite{We}, if and only if, for all nonzero integers $h$,
$$\lim_{N\rightarrow\infty}\frac{1}{N}\sum_{n=1}^N e^{2\pi ihx_n}=0.$$

In this paper, we first set up a result resembling that of Hooley's under ideal theoretic settings. Before doing so, we need some notations. Let $L$ be a number field and $\mathcal{O}$ be its ring of algebraic integers. For $\alpha\in L$, let $\eta$ be any positive integer such that   $\eta\alpha\in\mathcal{O}$. Then $\alpha$ is integral over $\mathbb{Z}[\frac{1}{\eta}]$. We consider the extension of Dedekind domains $\mathcal{O}[\frac{1}{\eta}]/\mathbb{Z}[\frac{1}{\eta}]$ instead of $\mathcal{O}/\mathbb{Z}$ in this article, so that $\alpha$ can be treated exactly as an algebraic integer.

For a prime ideal $\mathfrak{p}$ of $\mathcal{O}[\frac{1}{\eta}]$, let $p$ be the prime number such that
$$\mathfrak{p}\cap\mathbb{Z}[1/\eta]=p\mathbb{Z}[1/\eta].$$
The inertia degree of $\mathfrak{p}$ is defined to be the extension degree of the residue fields
$$[\mathcal{O}[1/\eta]/\mathfrak{p}:\mathbb{Z}
[1/\eta]/(p)]
=\frac{\log |\mathcal{O}[\frac{1}{\eta}]/\mathfrak{p}|}
{\log p},$$
and the ramification index is the largest integer $e$ such that $\mathfrak{p}^e|p$. A prime ideal $\mathfrak{p}$ is called unramified if its ramification index equals one, otherwise it is ramified. For any nontrivial ideal $\mathfrak{a}\subset\mathcal{O}[\frac{1}{\eta}]$, we call $\mathfrak{a}$ unramified if all its prime factors are unramified. Also we define the inertia degree of $\mathfrak{a}$ to be
$$\mathfrak{i}(\mathfrak{a})
:=\frac{\log|\mathcal{O}[\frac{1}{\eta}]/\mathfrak{a}|}
{\log|\mathbb{Z}[\frac{1}{\eta}]
/(\mathfrak{a}\cap\mathbb{Z}[\frac{1}{\eta}])|}.$$
Let
$\mathfrak{N}(\mathfrak{a}):=[\mathcal{O}[\frac{1}{\eta}]:\mathfrak{a}]$
be the absolute norm of $\mathfrak{a}$. The absolute norm is completely  multiplicative, i.e., for any ideals $\mathfrak{a}_1,\mathfrak{a}_2$,
$$\mathfrak{N}(\mathfrak{a}_1\mathfrak{a}_2)=\mathfrak{N}
(\mathfrak{a}_1)\mathfrak{N}(\mathfrak{a}_2).$$
For a principal ideal $(\alpha)$, its absolute norm is simply denoted by $\mathfrak{N}(\alpha)$. Note that $\mathfrak{N}(\alpha)$ does not necessarily equal to $|\prod_{\sigma}\sigma(\alpha)|$, where the $\sigma$  runs through all embeddings of $L$ to $\mathbb{C}$. But their quotient is an unite in $\mathbb{Z}[\frac{1}{\eta}]$.

We are interested in ideals satisfying $\mathfrak{i}(\mathfrak{a})=1$, or equivalently,
$$\mathcal{O}[1/\eta]/\mathfrak{a}\cong\mathbb{Z}
[1/\eta]/(\mathfrak{a}\cap\mathbb{Z}[1/\eta]),$$
which are called ideals of {\it inertia degree one} (or simply degree one). Let $n$ be the least positive integer prime to $\eta$ contained in the ideal $\mathfrak{a}\cap\mathbb{Z}[\frac{1}{\eta}]$. Then $n$ generates $\mathfrak{a}\cap\mathbb{Z}[\frac{1}{\eta}]$ and
$$\mathbb{Z}[1/\eta]/(\mathfrak{a}\cap\mathbb{Z}[1/\eta])
\cong\mathbb{Z}/n\mathbb{Z}.$$
So if $\mathfrak{a}$ is of degree one, then $n=\mathfrak{N}(\mathfrak{a})$ and there is a unique integer in
$[0,\mathfrak{N}(\mathfrak{a}))$, denoted by $\alpha(\mathfrak{a})$, such that $\alpha-\alpha(\mathfrak{a})\in \mathfrak{a}$, or say,
$$\alpha\equiv\alpha(\mathfrak{a})\mod \mathfrak{a}.$$
The following result relates the degree one ideals with the roots of a polynomial congruence.

\begin{thm}\label{thm1}
Let $\mathfrak{a}\subset\mathcal{O}[\frac{1}{\eta}]$ be an ideal, and
$$\mathfrak{a}=\mathfrak{p}_1^{e_1}\cdots\mathfrak{p}_r^{e_r}$$
be its unique factorization into prime ideals. Then
$\mathfrak{a}$ is of degree one if and only if the following two are true:

{\rm (i)} for $1\le i\le r$, each $\mathfrak{p}_i$ is of degree one, and $e_i=1$ if $\mathfrak{p}_i$ is ramified;

{\rm (ii)} $\gcd(\mathfrak{N}(\mathfrak{p}_i),\mathfrak{N}(\mathfrak{p}_j))=1$ for  $1\le i<j\le r$.

Furthermore, let $f$ be the primitive minimal polynomial of $\alpha$ over $\mathbb{Z}$ and $D_f$ be the discriminant of $f$. Suppose that $L=\mathbb{Q}(\alpha)$. If $n$ is prime to $\eta D_f$, then there is an one to one correspondence between the roots $v$ of $f(x)\equiv 0 \mod n$ and ideals $\mathfrak{a}$ of inertia degree one with $\mathfrak{N}(\mathfrak{a})=n$ given by
$$v\mapsto \mathfrak{a}=(\alpha-v,n)\ and\
\mathfrak{a}\mapsto v=\alpha(\mathfrak{a}).$$
\end{thm}

Let $S_{\eta}$ be the set of all ideals of degree one in $\mathcal{O}[\frac{1}{\eta}]$ and $S_{\eta}(n)$ be the set of elements of $S_{\eta}$ with absolute norm $n$. For a fixed $n$ prime to $\eta D_f$, one derives from the correspondence in Theorem 1.1 that $$\Big\{\frac{\alpha(\mathfrak{a})}{\mathfrak{N}(\mathfrak{a})}:
\mathfrak{a}\in S_{\eta}(n)\Big\}
=\Big\{\frac{v}{n}: 0\le v<n, f(v)\equiv0\mod n\Big\}.$$
So we take $\{\frac{\alpha(\mathfrak{a})}{\mathfrak{N}(\mathfrak{a})}:
\mathfrak{a}\in S_{\eta}\}$ as the substitute of the set of ratios
$\{v/n: n>0, 0\le v<n, f(v)\equiv 0\mod n\}$ under ideal theoretical settings. It's then reasonable to expect that the sequence of ratios $\frac{\alpha(\mathfrak{a})}{\mathfrak{N}(\mathfrak{a})}$ is also uniformly distributed as $\mathfrak{N}(\mathfrak{a})$ are ascending. Note that we supposed $L=\mathbb{Q}(\alpha)$ in Theorem 1.1. While the sequence of ratios $\frac{\alpha(\mathfrak{a})}{\mathfrak{N}(\mathfrak{a})}$ is well defined for any $\alpha\in\mathcal{O}[\frac{1}{\eta}]$, it is natural to ask if the sequence is uniformly distributed without this restriction.  Actually, we have the ideal theoretical parallel of Hooley's theorem.

\begin{thm}\label{thm2}
Let $\eta$ be an positive integer. For any irrational $\alpha\in\mathcal{O}[\frac{1}{\eta}]$, the sequence of the ratios
$$\{\alpha(\mathfrak{a}_i)/\mathfrak{N}(\mathfrak{a}_i)\}_{i=1}^{\infty}$$
is uniformly distributed, where $\mathfrak{a}_i$ runs through all elements in $S_{\eta}$ with
$$\mathfrak{N}(\mathfrak{a}_i)\le \mathfrak{N}(\mathfrak{a}_{i+1}).$$
\end{thm}

For the rest of the paper, when we refer to a sequence of ratios, we tacitly assume it is an arrangement of the indicated ratios such that the denominators of the ratios are ascending. By Wyel's criterion, to prove Theorem 1.2 is equivalent to prove
$$\sum_{n\le x}\sum_{\mathfrak{a}\in S_{\eta}(n)}
\exp(2\pi ih\alpha(\mathfrak{a})/n)=o(x),\ \forall h\neq0.$$
The method we used to get this equation is largely equivalent to the one Hooley used in \cite{Ho}. We actually get a little more in section 3:
$$\sum_{n\le x}|\sum_{\mathfrak{a}\in S_{\eta}(n)}
\exp(2\pi ih\alpha(\mathfrak{a})/n)|=o(x),\ \forall h\neq0.$$
This gives us a generalization of Theorem 1.2 (See Theorem 3.8). If one goes further in this direction, we can consider the distribution of the sequence
$$\{\alpha(\mathfrak{a})/\mathfrak{N}(\mathfrak{a}):
\mathfrak{a}\in S' \} \leqno(1.1)$$
for some subset $S'\subset S_{\eta}$. So far the following $S'$ are particularly interested to us:

(a). $S'$ is the set of all prime ideals of inertia degree one;

(b). $S'$ is the set of all ideals $\mathfrak{a}\in S_{\eta}$, where   $\mathfrak{N}(\mathfrak{a})\in\{n^k: n>0 \}$ for given positive integer $k$;

(c). $S'=\{\mathfrak{a}^k:k>0\}$ for fixed $\mathfrak{a}\in S_{\eta}$.

The first one relate to the question of distribution of roots of polynomial congruences to prime moduli, for quadratic case, see \cite{DFI} and \cite{To}.  By assuming the Bouniakowsky conjecture, Foo \cite{Fo} proved the density of the roots to prime moduli.
The later two are related to the distribution of digits of $n$-adic expansions of an given algebraic number, which will be discussed later. \\

Theorem 1.2 is not just a restatement of Hooley's result in language of ideals. In fact, a generalization of Hooley's theorem can be derived from Theorem 1.2.

\begin{thm}
Let $f_1,...,f_r\in\mathbb{Z}[x]$ be primitive irreducible polynomials of degree larger than 1 with discriminants $D_1,...,D_r$ respectively. Suppose that $D_i$'s are pairwisely coprime. Then the sequence of the r-tuples of ratios $(v_1/n,...,v_r/n)$, where $n>0$ runs through all positive integers and $v_i$ runs through roots of $f_i(x)\equiv0\mod n$ for each $i$, is uniformly distributed in $[0,1]^r$.
\end{thm}

When $r=1$, this is Hooley's theorem. Zehavi \cite{Ze} recently prove the case $r = 2$ without the restriction that
$D_1, D_2$ are coprime.  We outline the idea of proving Theorem 1.3 briefly. Let $\alpha_i$ be a root of $f_i$ for $1\le i\le r$ and $L=\mathbb{Q}(\alpha_1,...,\alpha_r)$. By Weyl's criterion for higher  dimensional sequences, the Theorem 1.2 can be extend to the uniformity of  the sequence
$$\{(\alpha_1(\mathfrak{a})/\mathfrak{N}(\mathfrak{a}),...,
(\alpha_r(\mathfrak{a})/\mathfrak{N}(\mathfrak{a}))\}_
{\mathfrak{a}\in S_{\eta}} \leqno(1.2)$$
for suitable $\eta$. Meanwhile, the one to one correspondence given by  Theorem 1.1 can be extended, and a following equation can be established
$$\{(\alpha_1(\mathfrak{a}),...,\alpha_r(\mathfrak{a})):\mathfrak{a}\in S_{\eta}(n)\}=\{(v_1,...,v_r): f_i(v_1)\equiv 0\mod n\}$$
for positive integers $n$ prime to $\eta D_i$ for $1\le i\le r$. Then from this equation, we can prove Theorem 1.3 by the uniformity of the sequence (1.2). \\

Another topic in this article concerns the digits of $n$-adic expansions of an given algebraic number. It has its analogue in the study of the distribution of digits of $n$-ary expansion of an real number. Before proceeding, we give some backgrounds. For a real number $\gamma$, let $\{\gamma\}$ denote its fractional part. Let
$$\{\gamma\}=\sum_{l=1}^{\infty}\frac{a_l(n)}{n^l},\ 0\le a_l(n)<n
\leqno(1.3)$$
be the $n$-ary expansion of $\{\gamma\}$. Recall that $\gamma$ is normal in base $n$ if, for every positive integer $m$ and each word $\omega\in\{0,1,...,n-1\}^m$,
$$\lim_{N\rightarrow\infty}\frac{|\{l<N:(a_{l+1}(n),...,a_{l+m}(n))
=\omega\}|}{N}=\frac{1}{n^m}.$$
The real number $\gamma$ is called absolutely normal if it is normal in every base $n\ge 2$. Borel \cite{Bo1} showed that almost all real numbers are absolutely normal with respect to Lebesgue measure. Furthermore Borel \cite{Bo2} conjectured that all irrational algebraic real numbers are absolutely normal. It is known that $\gamma$ is normal in base $n$ if and only if the sequence $\{\{n^l\gamma\}\}_{l=1}^{\infty}$ is uniformly distributed.

It is interesting to compare this with a result of Weyl. Let $f(x)\in\mathbb{R}[x]$. Weyl \cite{We} showed that the sequence $\{\{f(n)\}\}_{n=1}^{\infty}$ is uniformly distributed, provided that at least one of the coefficients of nonconstant term of $f(x)$ is irrational. Specially, for any given positive integer $l$ and irrational real number $\gamma$, the sequence $\{\{n^l\gamma\}\}_{n=1}^{\infty}$ is uniformly distributed. Let $a_{l+1}(n)$ be given as in (1.3). Then
$$\{n^l\gamma\}=\frac{a_{l+1}(n)}{n}+O\Big(\frac{1}{n}\Big).$$
So $\{\{n^l\gamma\}\}_{n=1}^{\infty}$ is uniformly distributed if and only if the sequence $\{\frac{a_{l+1}(n)}{n}\}_{n=1}^{\infty}$ is uniformly distributed.

Now we pose two similar questions concerning $n$-adic expansions of an algebraic number. Let $\alpha$ be an irrational algebraic number and $f(x)$ be its primitive minimal polynomial over $\mathbb{Z}$. For an integer $n>1$, let $\mathbb{Z}_n$ denote the $n$-adic completion of $\mathbb{Z}$, i.e., the projective limit
$$\varprojlim_{l} \mathbb{Z}/n^l\mathbb{Z}.$$
We say that $a_0+a_1n+a_2n^2+\cdots\in\mathbb{Z}_n$, where
$0\le a_i<n$ for all $i\ge0$, is an $n$-adic expansion of $\alpha$ if
$$f(a_0+a_1n+a_2n^2+\cdots)=0.$$
For an arbitrary positive integer $n$, $\alpha$ may have none or multiple $n$-adic expansions. Let $\rho_{\alpha}(n)$ be the number of different $n$-adic expansions of $\alpha$. Let
$$\sum_{l=0}^{\infty}a_l(n,m)n^l,\ {\rm where}\
1\le m\le \rho_{\alpha}(n),0\le a_l(m,n)<n,$$
denote all the different $n$-adic expansions of $\alpha$. Arrange the ratios $\frac{a_l(n,m)}{n}$ into a sequence by taking the lexicographical order of $(n,m)$. This sequence is related to the sequence (1.1) with $S'$ given by (b). We will prove the uniformity for $l=0$ as a consequence of Theorem 1.2.

We also consider $n$-adic analogue of normal numbers. An $n$-adic number
$\sum_{l=0}^{\infty}a_ln^l$ is called normal if for any $m>1$ and $\omega\in\{0,...,n-1\}^m$,
$$\lim_{N\rightarrow\infty}\frac{|\{l<N|\omega=(a_{l},...,a_{l+m-1})\}|}
{N}=\frac{1}{n^m}.$$
Similarly, we pose the $n$-adic normal number conjecture:
\begin{conj}
Let $\alpha$ be an irrational algebraic number and $n>1$ be an integer. Then all $n$-adic expansions of $\alpha$, if exist, are normal.
\end{conj}
An equivalent form of this conjecture is given in the last section. To support this conjecture, we show that almost all elements in $\mathbb{Z}_n$ are n-adically normal with respect to the Haar measure.

%

\section{Preliminaries}
We list some basic facts in algebraic number theory which will be used and prove Theorem 1.1 in this section. First we introduce the Dedekind factorization theorem.

\begin{lem}\cite{Ne}
Let $\mathcal{A}$ be a Dedekind domain and $F$ be its fraction field. Let $K$ be a finite separable extension of $F$ and $\mathcal{B}$ be the integral closure of $\mathcal{A}$ in $K$. Suppose that $\alpha\in\mathcal{B}$ satisfying $K=F(\alpha)$ and $f_{\alpha}(x)$ is its monic minimal polynomial over $F$. Denote by $\mathfrak{f}$ the conductor of $\mathcal{A}[\alpha]$ in $\mathfrak{B}$. For a prime ideal $\mathfrak{p}$ of $\mathcal{A}$ which is relatively prime to $\mathfrak{f}$, let
$$\bar{f}(x)=\bar{f}_1(x)^{e_1}\cdots \bar{f}_r(x)^{e_r}$$
be the factorization of $\bar{f}(x)=f_{\alpha}(x) \mod\mathfrak{p}$ into irreducibles  $\bar{f}_i(x)=f_i(x)\mod\mathfrak{p}$ over $\mathcal{A}/\mathfrak{p}$. Then
$$\mathfrak{P}_i=\mathfrak{p}\mathcal{B}+f_i(\alpha)\mathcal{B},
\ i=1,...,r$$
be different prime ideals of $\mathcal{B}$ above $\mathfrak{p}$. The inertia degree of $\mathfrak{P}_i$ is the degree of $f_i(x)$ and one has
$$\mathfrak{p}=\mathfrak{P}_1^{e_1}\cdots\mathfrak{P}_r^{e_r}.$$
\end{lem}

Recall that the conductor of $\mathcal{A}[\alpha]$ in $\mathcal{B}$ equals $\{\beta\in\mathcal{A}[\alpha]: \beta\mathcal B\in\mathcal{A}[\alpha]\}$. It is also the greatest ideal of $\mathcal{B}$ contained in $\mathcal{A}[\alpha]$. Let $k:=[K:F]$ and $D_{\alpha}$ be the discriminant of $1,\alpha,...,\alpha^{k-1}$ with respect to $K/F$. Then by Chapter 1, section 2, Lemma 2.9 of \cite{Ne}, $D_{\alpha}\mathcal{B}\subset\mathcal{A}[\alpha]$. So $D_{\alpha}$ is in the conductor and any ideal of $\mathcal{B}$ prime to $D_{\alpha}$ is prime to the conductor. We may use $D_{\alpha}$ instead of the conductor in the Lemma 2.1. Let $\mathcal{A}=\mathbb{Z}[\frac{1}{\eta}]$ and $\mathcal{B}=\mathcal{O}[\frac{1}{\eta}]$. Dedekind's theorem implies that there is a one to one correspond between degree one prime ideals above $p$ and roots of $f_{\alpha}(x)\mod p$ for each $p\nmid \eta D_f$. Our Theorem 1.1 generalize this fact to positive integers $n$ coprime to $\eta D_f$.

\begin{lem}\cite{Hi}
Let $K_1,K_2$ be two number fields whose discriminants are relatively prime and $K_1K_2$ be the composition of this two field. Then
$$[K_1K_2:\mathbb{Q}]=[K_1:\mathbb{Q}][K_2:\mathbb{Q}]$$
and
$$D_{K_1K_2}=D_{K_1}^{[K_2:\mathbb{Q}]}D_{K_2}^{[K_1:\mathbb{Q}]}$$
where $D_{K_1},D_{K_2}$ and $D_{K_1K_2}$ represent discriminants of $K_1, K_2$ and $K_1K_2$ respectively.
\end{lem}

We also need Landau's Prime Ideal Theorem and the Wiener-Ikehara Tauberian Theorem for Dirichlet L-series (cf. \cite{MV} $\S8.3$).
\begin{lem}\cite{La} For any number field,
$$\sum_{\mathfrak{N}(\mathfrak{p})\le x}1=\frac{x}{\log x}
+O(\frac{x}{\log^2x}).$$
\end{lem}

\begin{lem}
Let $F(s)=\sum_{n=1}^{\infty}\frac{a_n}{n^s}$ where $a_n$ are nonnegative real numbers for all $n$. If there is a $c\ge 0$ such that $F(s)-c/(s-1)$ converges to an analytic function in $\Re(s)\ge 1$, then
$$\sum_{n\le x}a_n=cx+o(x).$$
\end{lem}

We now prove Theorem 1.1.

{\it Proof of Theorem \ref{thm1}.}
We begin with an easy case: $\mathfrak{a}=\mathfrak{p}^k$, where $\mathfrak{p}$ is a prime ideal of $\mathcal{O}[\frac{1}{\eta}]$ and $p$ is the prime number with $p\mathbb{Z}[\frac{1}{\eta}]=\mathfrak{p}\cap\mathbb{Z}[\frac{1}{\eta}]$. Then $p^k\in\mathfrak{p}^k\cap\mathbb{Z}[\frac{1}{\eta}]$ and so
$p^k\mathbb{Z}[\frac{1}{\eta}]\subset\mathfrak{p}^k
\cap\mathbb{Z}[\frac{1}{\eta}]$. Therefore
$$|\mathcal{O}[1/\eta]/\mathfrak{p}^k|
=|\mathcal{O}[1/\eta]/\mathfrak{p}|^k\ge p^k
=|\mathbb{Z}[1/\eta]/p^k\mathbb{Z}[1/\eta]|
\ge |\mathbb{Z}[1/\eta]/(\mathfrak{p}^k\cap\mathbb{Z}[1/\eta])|.$$
It is known that $|\mathcal{O}[\frac{1}{\eta}]/\mathfrak{p}|^k=p^k$ if and only if $\mathfrak{p}$ is of degree one, and
$$\mathbb{Z}[1/\eta]/p^k\mathbb{Z}[1/\eta]\cong
\mathbb{Z}[1/\eta]/(\mathfrak{p}^k\cap\mathbb{Z}[1/\eta])$$
if and only if $k=1$ or $\mathfrak{p}$ is unramified.

For $\mathfrak{a}=\mathfrak{p}_1^{e_1}\cdots\mathfrak{p}_r^{e_r}$, we have
$$|\mathcal{O}[1/\eta]/\mathfrak{a}|=\prod_{i=1}^{r}
|\mathcal{O}[1/\eta]/\mathfrak{p}_i^{e_i}|\ge\prod_{i=1}^{r}
|\mathbb{Z}[1/\eta]/(\mathfrak{p}_i^{e_i}\cap\mathbb{Z}[1/\eta])|
\ge|\mathbb{Z}[1/\eta]/(\mathfrak{a}\cap\mathbb{Z}[1/\eta])|.$$
On the one hand,
$$\prod_{i=1}^{r}|\mathcal{O}[1/\eta]/\mathfrak{p}_i^{e_i}|
=\prod_{i=1}^{r}|\mathbb{Z}[1/\eta]/(\mathfrak{p}_i^{e_i}\cap\mathbb{Z})|$$
if and only if
$$|\mathcal{O}[1/\eta]/\mathfrak{p}_i^{e_i}|
=|\mathbb{Z}[1/\eta]/(\mathfrak{p}_i^{e_i}\cap\mathbb{Z})|$$
for all $1\le i\le r$, if and only if (i) is true by above discussion on the case $\mathfrak{a}=\mathfrak{p}^k$. On the other hand,
$$\prod_{i=1}^{r}|\mathbb{Z}[1/\eta]/(\mathfrak{p}_i^{e_i}
\cap\mathbb{Z}[1/\eta])|
=|\mathbb{Z}[1/\eta]/(\mathfrak{a}\cap\mathbb{Z}[1/\eta])|$$
if and only if
$$\prod_{i=1}^{r}(\mathfrak{p}_i^{e_i}\cap\mathbb{Z})[1/\eta]|
=(\prod_{i=1}^{r}\mathfrak{p}_i^{e_i}\cap\mathbb{Z}[1/\eta]).$$
Claim that the later equation happens if and only if (ii) is true. Indeed,
$$(\prod_{i=1}^r \mathfrak{p}_i^{e_i})\cap\mathbb{Z}[1/\eta]
=(\bigcap_{i=1}^r \mathfrak{p}_i^{e_i})\cap\mathbb{Z}[1/\eta]
=\bigcap_{i=1}^r (\mathfrak{p}_i^{e_i}\cap\mathbb{Z}[1/\eta])
\supset\prod_{i=1}^r (\mathfrak{p}_i^{e_i}\cap\mathbb{Z}[1/\eta]),$$
and the equal happens if and only if $(\mathfrak{p}_i^{e_i}\cap\mathbb{Z}[\frac{1}{\eta}])$ are pairwisely coprime, equivalently, (ii) is true. Hence we conclude that $\mathfrak{a}$ is an degree one ideal if and only if both (i) and (ii) are true. \\

It is left to prove the furthermore part of Theormen 1.1. Let $\mathfrak{a}$ be an ideal of inertia degree one with $\mathfrak{N}(\mathfrak{a})=n$. Set $v=\alpha(\mathfrak{a})$. Then $v-\alpha\in\mathfrak{a}$ and so $$\frac{f(v)}{c}=\prod_{\sigma}(v-\sigma(\alpha))\in\mathfrak{a},$$
where $c$ is the leading coefficient of $f(x)$ and $\sigma$ runs through the embedding of $L$ to $\mathbb{C}$. Since $f(x)/c\in\mathbb{Z}[\frac{1}{\eta}][x]$, $c$ is an unite in $\mathbb{Z}[\frac{1}{\eta}]$. It follows from
$$\prod_{\sigma}(v-\sigma(\alpha))\in\mathfrak{a}\cap\mathbb{Z}[1/\eta]
=\mathfrak{N}(\mathfrak{a})\mathbb{Z}[1/\eta]$$
that $f(v)\equiv 0\mod n$. That is, $\alpha(\mathfrak{a})$ is a root of $f(x)\equiv0\mod n$. So we have the map $\mathfrak{a}\mapsto v=\alpha(\mathfrak{a})$.

Conversely let $(n,\eta D_f)=1$ and $v$ is a root of $f(x)\equiv0\mod n$. Since the leading coefficient $c$ of $f(x)$ is a unite in $\mathbb{Z}[\frac{1}{\eta}]$, $v$ is a root of the congruence
$$f(x)/c\equiv 0\mod n\mathbb{Z}[1/\eta].$$
Let $d:=\deg f$ and $D_{\alpha}$ be the discriminant of $1,\alpha,...,\alpha^{d-1}$. Since $D_f$ equals the product of $D_{\alpha}$ and a power of $c$, $D_f$ and $D_{\alpha}$ generate the same ideal in $\mathbb{Z}[\frac{1}{\eta}]$. Write  $n=p_1^{e_1}\cdots p_r^{e_r}$. Then $p_i\nmid D_{\alpha}$ in $\mathbb{Z}[\frac{1}{\eta}]$. Take $\mathcal{A}=\mathbb{Z}[\frac{1}{\eta}], \mathcal{B}=\mathcal{O}[\frac{1}{\eta}]$ and $f_{\alpha}=f/c$ in Lemma 2.1. Applying the Dedekinds factorization theorem to each $p_i\mathbb{Z}[\frac{1}{\eta}]$, we obtain that $(\alpha-v,p_i)$ is a prime ideal of degree one, where $(\alpha-v,p_i)$ is the ideal of $\mathcal{O}[\frac{1}{\eta}]$ generated by $\alpha-v$ and $p_i$.

Denote by $\mathfrak{p}_i=(\alpha-v,p_i)$. Note that any prime ideal other than $\mathfrak{p}_i$ above $p_i$ does not divide $\alpha-v$. In fact, if $\mathfrak{q}|p_i$ is a prime ideal different from $\mathfrak{p}_i$ that  divides $(\alpha-v)$, then $\mathfrak{q}\mathfrak{p_i}|(\alpha-v,p_i)=\mathfrak{p}_i$, a contradiction.
So $(\alpha-v)=\mathfrak{p}_i^t\mathfrak{b}$ for some positive integer $t$ and ideal $\mathfrak{b}$ with $\mathfrak{b}+p_i\mathcal{O}[\frac{1}{\eta}]=\mathcal{O}[\frac{1}{\eta}]$, and hence
$$(\alpha-v, p_i^{e_i})=\mathfrak{p}_i^{\min\{e_i,t\}}.$$
We show $t\ge e_i$. Since $f(v)\equiv 0\mod p_i^{e_i}$, it follows that $p_i^{e_i}$ divides $\mathfrak{N}(\alpha-v)$ in the ring $\mathbb{Z}[\frac{1}{\eta}]$. Together with $(\mathfrak{N}(\mathfrak{b}),p_i)=1$, this gives us $p_i^{e_i}|\mathfrak{N}(\mathfrak{p}_i^t)=p^t$. Thus $t\ge e_i$.
Therefore $(\alpha-v, p_i^{e_i})=\mathfrak{p}_i^{e_i}$.

Since $(\alpha-v,n)\subset(\alpha-v,p_i^{e_i})$ for each $i$ with $1\le i\le r$, one derives that
$$(\alpha-v,n)\subset(\alpha-v,p_1^{e_1})\cdots(\alpha-v,p_r^{e_r}).$$
By comparing the generators of ideals of each side, it is easy to see that $$(\alpha-v,p_1^{e_1})\cdots(\alpha-v,p_r^{e_r})\subset(\alpha-v,n).$$
So the two ideals are equal. Since the product $(\alpha-v,p_1^{e_1})\cdots(\alpha-v,p_r^{e_r})$ satisfies both (i) and (ii) of Theorem 1.1, we conclude that $(\alpha-v,n)$ is an ideal of degree one with norm $n$. Hence we set up the map $v\mapsto\mathfrak{a}=(\alpha-v,n)$ and it is the inverse of the map $\mathfrak{a}\mapsto v=\alpha(\mathfrak{a})$. This proves the furthermore part and ends the proof of Theorem 1.1. \hfill$\Box$ \\

We end this section with some discussion about $S_{\eta}(n)$. It is easy to give the structure of $S_{\eta}(n)$ by Theorem 1.1.

\begin{cor}
One has
$$S_{\eta}(p^e)=\{\mathfrak{p}^e: \mathfrak{p}\in S_{\eta}(p)\
\rm{is\ unramified}\},\ \forall e\ge 2.$$
For $n=p_1^{e_1}\cdots p_r^{e_r}$,
$$S_{\eta}(n)=S_{\eta}(p_1^{e_1})\times\cdots\times S_{\eta}(p_r^{e_r})
=\{\mathfrak{p}_1^{e_1}\cdots\mathfrak{p}_r^{e_r}
: \mathfrak{p}_i\in S_{\eta}(p_i^{e_i}), 1\le i\le r\}.$$
\end{cor}
We are also interested in the set
$\{\alpha(\mathfrak{a}): \mathfrak{a}\in S_{\eta}(n)\}$ for $\alpha\in \mathcal{O}[\frac{1}{\eta}]$ which does not necessarily satisfy  $L=\mathbb{Q}(\alpha)$. We have
$$\prod_{\sigma}(\sigma(\alpha)-\alpha(\mathfrak{a}))
\in \mathfrak{a}\cap\mathbb{Z}[1/\eta],$$
where $\sigma$ runs through all embeddings of $\mathbb{Q}(\alpha)$ to $\mathbb{C}$. Then $f(\alpha(\mathfrak{a}))\equiv 0\mod n$, where $f(x)$ denotes the primitive minimal polynomial of $\alpha$ over $\mathbb{Z}$. Hence
$$\{\alpha(\mathfrak{a}): \mathfrak{a}\in S_{\eta}(n)\}
\subset\{v:f(v)\equiv 0\mod n\}.$$
When $L\neq \mathbb{Q}(\alpha)$, there may exist more than one $\mathfrak{a}\in S_{\eta}(n)$ such that $\alpha(\mathfrak{a})=v$. For example, let $p\nmid\eta$ be a prime number which splits completely in $L$ and dose not divide the discriminant of $f$. Then
$$\{\alpha(\mathfrak{a}): \mathfrak{a}\in S_{\eta}(p)\}
=\{v:f(v)\equiv 0\mod p\},$$
and for each root $v$, the number of $\mathfrak{p}\in S_{\eta}(p)$ with $\alpha(\mathfrak{p})=v$ is exactly $[L:\mathbb{Q}(\alpha)]$.

\section{distribution of residues of $\alpha$ modulo degree one ideals}

In this section, we focus our discussion on an fixed positive integer $\eta$, algebraic number field $L$ with ring of algebraic integers $\mathcal{O}$, and irrational algebraic number $\alpha\in \mathcal{O}[\frac{1}{\eta}]$ with primitive minimal polynomial $f(x)$ over $\mathbb{Z}$. Let $l:=[L:\mathbb{Q}]$, $d:=\deg f$ and $D_f$ be the discriminant of $f$. Let $S_{\eta}$ denote the set of degree one ideals of $\mathcal{O}[\frac{1}{\eta}]$ and $S_{\eta}(n)\subset S_{\eta}$ be the set of degree one ideals with norm $n$. Set
$$\rho_{\eta}(n):=|S_{\eta}(n)|.$$
Then $\rho_{\eta}$ is a function on $n$ with $(n,\eta)=1$. We extend it to all positive integers by setting $\rho_{\eta}(n):=0$ for $(n,\eta)>1$.
For $\eta=1$, $S_{\eta}$, $S_{\eta}(n)$ and $\rho_{\eta}(n)$ are simply denoted by $S, S(n)$ and $\rho(n)$. Since $\mathcal{O}[\frac{1}{\eta}]$ is the ring of fractions of $\mathcal{O}$ with respect to $\{\eta^n:n\ge 0\}$, the ideals of $\mathcal{O}[\frac{1}{\eta}]$ are in one to one correspondence with ideals of $\mathcal{O}$ that are prime to $\eta$. So we have
$$\rho_{\eta}(n)=\rho(n),\ \forall (n,\eta)=1.$$
Denote by $P$ the set of prime numbers which are prime to $\eta D_{f}$ and split completely in $L$. For a positive integer $n$, let $\omega(n)$ denote the number of different prime factors of $n$, and define
$$\omega_P(n):=|\{p\in P: p\mid n\}|.$$
Our main mission of this section is to prove Theorem 1.2. The first lemma is an immediately consequence of Corollary 2.5.

\begin{lem} We have

{\rm (i)}. $\rho_{\eta}(p^a)\le \rho_{\eta}(p)\le l$ for every $p$. Furthermore, $\rho_{\eta}(p^a)=\rho_{\eta}(p)$ if $p$ is unramified and $\rho_{\eta}(p)=l$ if $p$ split completely in $L$;

{\rm (ii)}. $\rho_{\eta}(n_1n_2)=\rho_{\eta}(n_1)\rho_{\eta}(n_2)$ if $n_1,n_2$ are relatively prime;

{\rm (iii)}. $\rho_{\eta}(n)\le l^{\omega(n)}$.
\end{lem}

As is usual in analytic number theory, we denote $e(x):=\exp(2\pi ix)$. For any integer $h$, let
$$\rho_{\eta}(h,n)=\sum_{\mathfrak{a}\in S_{\eta}(n)}
e\Big(\frac{h\alpha(\mathfrak{a})}{n}\Big).$$
Obviously $|\rho_{\eta}(h,n)|\le \rho_{\eta}(n)$ and $\rho_{\eta}(0,n)=\rho_{\eta}(n)$.

\begin{lem}
Let $n=n_1n_2$ with $(n_1,n_2)=1$. Suppose $\bar{n}_1, \bar{n}_2$ be positive integers with $n_1\bar{n}_1\equiv 1\mod n_2$ and $n_2\bar{n}_2\equiv 1\mod n_1$. Then
$$\rho_{\eta}(h,n)=\rho_{\eta}(h\bar{n}_2,n_1)
\rho_{\eta}(h\bar{n}_1,n_2).$$
\end{lem}

\begin{proof}
By Corollary 2.5, for each $\mathfrak{a}\in S_{\eta}(n)$, there exists unique
$$(\mathfrak{a}_1, \mathfrak{a}_2)\in
S_{\eta}(n_1)\times S_{\eta}(n_2)$$
such that $\mathfrak{a}=\mathfrak{a}_1\mathfrak{a}_2$. By definition,
$$\alpha(\mathfrak{a})\equiv \alpha(\mathfrak{a}_1)\mod n_1,\
\alpha(\mathfrak{a})\equiv \alpha(\mathfrak{a}_2)\mod n_2.$$
So we derive form the Chinese Remainder Theorem that
$$\alpha(\mathfrak{a})\equiv n_2\bar{n}_2\alpha(\mathfrak{a}_1)
+n_1\bar{n}_1\alpha(\mathfrak{a}_2) \mod n.$$
Therefore
\begin{align*}
\rho_{\eta}(h,n)=&\sum_{(\mathfrak{a}_1,\mathfrak{a}_2)
\in S_{\eta}(n_1)\times S_{\eta}(n_2)}
e\Big(\frac{hn_2\bar{n}_2\alpha(\mathfrak{a}_1)
+hn_1\bar{n}_1\alpha(\mathfrak{a}_2)}{n}\Big) \\
=&\rho_{\eta}(h\bar{n}_2,n_1)\rho_{\eta}(h\bar{n}_1,n_2).
\end{align*}
\end{proof}

\begin{lem}
Let $a,m$ be relatively prime positive integers. Then
$$\sum_{n\le x, \atop n\equiv a(m)}\rho(n)=\frac{c_1x}{\varphi(m)}+o(x),$$
where $\varphi$ is the Euler function and $c_1>0$ is a constant depending on $m$.
\end{lem}

\begin{proof}
Let $\chi$ be a Dirichlet character mod $m$. For each integral ideal $\mathfrak{a}$, define $$\chi\circ\mathfrak{N}(\mathfrak{a}):=\chi(\mathfrak{N}(\mathfrak{a})).$$
Then $\chi\circ\mathfrak{N}$ is a Hecke character mod $m\mathcal{O}$. Its  Hecke L-series, denoted by $L(\chi,s)$, is
$$L(\chi,s)=\sum_{\mathfrak{a}}\frac{\chi\circ\mathfrak{N}(\mathfrak{a})}
{\mathfrak{N}(\mathfrak{a})^s}=\prod_{\mathfrak{p}}\frac{1}{1-
\chi\circ\mathfrak{N}(\mathfrak{p})\mathfrak{N}(\mathfrak{p})^{-s}}.$$
It is analytic continuable to an entire function when $\chi$ is nontrivial, or to a meromorphic function with a simple pole at $s=1$ when $\chi$ is trivial.

Write $L(\chi,s)=\Phi_{\chi}(s)\Psi_{\chi}(s)$, where
$$\Phi_{\chi}(s):=\prod_{\mathfrak{p}\in S}\frac{1}{1-\chi\circ
\mathfrak{N}(\mathfrak{p})\mathfrak{N}(\mathfrak{p})^{-s}},\ \
\Psi_{\chi}(s):=\prod_{\mathfrak{p}\not\in S}\frac{1}{1-\chi\circ
\mathfrak{N}(\mathfrak{p})\mathfrak{N}(\mathfrak{p})^{-s}}.$$
Since $\mathfrak{N}(\mathfrak{p})\ge p^2$ for each $\mathfrak{p}\not\in S$, the infinite product of $\Psi_{\chi}(s)$ converges uniformly in any compact subset of the half plane $\Re(s)>1/2$. Hence $\Psi_{\chi}(s)$ is analytic and nonvanishing for all $s$ with $\Re(s)>1/2$. It follows that $1/\Psi_{\chi}(s)$ is analytic in the half plane $\Re(s)>1/2$. Therefore $\Phi_{\chi}(s)$ is analytic in $\Re(s)>1/2$ when $\chi$ is nontrivial, or meromorphic in $\Re(s)>1/2$ with a simple pole at $s=1$ when $\chi$ is trivial.

We then consider the analytic continuation of the Dirichlet L-series $$\sum_{n=1}^{\infty}\frac{\chi(n)\rho(n)}{n^s}=\prod_{p}\Big(1+
\frac{\chi(p)\rho(p)}{p^s}+\frac{\chi(p^2)\rho(p^2)}{p^{2s}}+\cdots\Big).$$
By $\mathfrak{N}(\mathfrak{p})=p$ for each $\mathfrak{p}\in S$, one has
$\Phi_{\chi}(s)=\prod_p(1-\chi(p)p^{-s})^{-\rho(p)}$. Since for large prime numbers $p$,
$$\log \Big(\frac{1}{1-\chi(p)p^{-s}}\Big)^{\rho(p)}-
\log\Big(1+\frac{\chi(p)\rho(p)}{p^s}+\frac{\chi(p^2)
\rho(p^2)}{p^{2s}}+\cdots\Big)=O\Big(\frac{\rho(p)}{p^{2s}}\Big),$$
it follows that
$\log \Psi_{\chi}(s)-\log(\sum_{n=1}^{\infty}\chi(n)\rho(n)/n^s)$
is analytic in $\Re(s)>1/2$. Taking exponentials, we obtain that
$\sum_{n=1}^{\infty}\chi(n)\rho(n)/n^s$ is continuable to an analytic function in $\Re(s)>1/2$ when $\chi$ is nontrivial, or to an meromorphic function with a simple pole at $s=1$ when $\chi$ is trivial.

Let $b$ be a positive integer with $ab\equiv 1\mod m$. Then
$$\sum_{n\ge 1 \atop n\equiv a (m)}\frac{\varphi(m)\rho(n)}{n^s}
=\sum_{n=1}^{\infty}\frac{\sum_{\chi}\chi(bn)\rho(n)}{n^s}
=\sum_{\chi}\sum_{n=1}^{\infty}\frac{\chi(bn)\rho(n)}{n^s},$$
where $\chi$ run through all Dirichlet characters mod $m$, is meromophic in $\Re(s)>1/2$ with a simple pole at $s=1$. Hence by the Ikehara's Tauberian Theorem,
$$\sum_{n\le x \atop n\equiv a (m)}\rho(n)=\frac{c_1x}{\varphi(m)}+o(x),$$
where $c_1$ is the residue of $\sum_{n=1}^{\infty}\chi(n)\rho(n)/n^s$ at $s=1$ when $\chi$ is trivial. This end the proof of Lemma 3.3.
\end{proof}

When $\chi$ is trivial,
$$\sum_{n=1}^{\infty}\frac{\chi(n)\rho(n)}{n^s}
=\sum_{n=1 \atop (n,m)=1}^{\infty}\frac{\rho(n)}{n^s}
=\sum_{n=1}^{\infty}\frac{\rho(n)}{n^s}
\prod_{p|m}(1+\rho(p)/p^s+\rho(p^2)/p^2+\cdots)^{-1}.$$
So $c_1=c\prod_{p|m}(1+\rho(p)/p+\rho(p^2)/p^2+\cdots)^{-1}$, where $c$ is the residue of $\sum_{n=1}^{\infty}\rho(n)/n^s$ at $s=1$. This shows how the constant $c_1$ depends on $m$. It also gives us
$$\sum_{n\le x \atop (n,m)=1}\rho(n)\sim c_1x.$$
Since $\rho_{\eta}(n)=\rho(n)$ for positive integers $n$ with $(\eta,n)=1$, we have
$$\sum_{n\le x \atop (n,m)=1}\rho_{\eta}(n)
=\sum_{n\le x \atop (n,\eta m)=1}\rho(n)
=cx/\prod_{p|\eta m}(1+\rho(p)/p+\rho(p^2)/p^2+\cdots).$$

\begin{lem}
There are constants $c_2$ and $c_3$ such that:

$$\sum_{p\le x}\rho_{\eta}(p)=\frac{x}{\log x}+O\Big(\frac{x}{\log^2x}
\Big), \leqno({\rm i}).$$

$$\sum_{p\le x}\frac{\rho_{\eta}(p)}{p}=\log\log x+c_2+O
(1/\log x), \leqno({\rm ii}).$$

$$\prod_{p\le x}(1+\frac{\rho_{\eta}(p)}{p})=c_3\log x(1+O(1/\log x)). \leqno({\rm iii}).$$
\end{lem}

\begin{proof}
We only prove the lemma for $\eta=1$, the general case follows immediately by $\rho_{\eta}(p)=\rho(p)$ for all $p\nmid\eta$.

(i). One has
$$\sum_{\mathfrak{N}(\mathfrak{p})\le x}1
=\sum_{t=1}^{\infty}\sum_{p\le x^{1/t}}
|\{\mathfrak{p}:\mathfrak{N}(\mathfrak{p})=p^t\}|
=\sum_{p\le x}\rho(p)+O(x^{1/2}/\log x).$$
So (i) follows from the Prime Ideal Theorem. \\

(ii). Write
$$A(x):=\sum_{p\le x}\rho(p)=x/\log x+R(x),$$
and $a_n:=A(n)-A(n-1)$. By the summation by parts formula,
\begin{align*}
\sum_{p\le x}\frac{\rho(p)}{p}
=&\sum_{2<n\le x}\frac{a_n}{n}+\frac{\rho(2)}{2}
=\frac{A(x)}{x}+\int_{2}^{x}\frac{A(t)}{t^2}dt \\
=&\frac{1}{\log x}+O\Big(\frac{1}{\log^2 x}\Big)
+\int_{2}^x\frac{1}{t\log t}dt+\int_{2}^{x}
\frac{R(t)}{t^2}dt \\
=&\log\log x-\log\log 2+\int_{2}^{x}\frac{R(t)}
{t^2}dt+O\Big(\frac{1}{\log x}\Big).
\end{align*}
To prove (ii), it remains to show
$$\int_{2}^{x}\frac{R(t)}{t^2}dt={\rm constant}
+O\Big(\frac{1}{\log x}\Big).$$
In fact, since $R(t)=O(t/\log^2t)$ by (i), the improper integral $\int_{2}^{\infty}\frac{R(t)}{t^2}dt$ converges. Denote by
$I:=\int_{2}^{\infty}\frac{R(t)}{t^2}dt$. Then
$$\int_{2}^{x}\frac{R(t)}{t^2}dt=I-\int_{x}^{\infty}
\frac{R(t)}{t^2}dt=I+O\Big(\int_{x}^{\infty}\frac{1}
{t\log ^2t}dt\Big)=I+O\Big(\frac{1}{\log x}\Big)$$
as desired. \\

(iii). Using the expansion of $\log(1+x)\ (|x|<1)$, one gets
$$\sum_{p\le x}\log\Big(1+\frac{\rho(p)}{p}\Big)
=\sum_{p\le l}\log\Big(1+\frac{\rho(p)}{p}\Big)
+\sum_{l<p\le x}\frac{\rho(p)}{p}+\sum_{l<p\le x}r_p,$$
where
$$r_p:=\sum_{n\ge2}\frac{(-1)^{n-1}\rho(p)^n}{np^n}=O(1/p^2).$$
The infinite sum $\sum_{p>l}r_p$ converges absolutely to some real number, say, $R$. Then
$$\sum_{l<p\le x}r_p=R+\sum_{p>x}O(1/p^2)=R+O(1/x).$$
By (ii) we can obtain
$$\sum_{p\le x}\log\Big(1+\frac{\rho(p)}{p}\Big)
=\log\log x+{\rm constant}+O(1/\log x),$$
from which (iii) follows by taking exponentials on both sides.
\end{proof}

\begin{lem}
Let $N$ be the normal closure of $L/\mathbb{Q}$. Then there is a constant $c_4$ that
$$\prod_{p\le x \atop p\in P}\Big(1+\frac{\rho_{\eta}(p)}{p}\Big)
=c_4\log^{1/[N:L]}x+O(1).$$
\end{lem}

\begin{proof}
It is enough to show the lemma for $\eta=1$. Let $\rho_N(n)$ denote number of degree one ideals of norm $n$ in $N$.  Since $N$ is normal, $\rho_N(p)=0$ for all but finitely many prime numbers $p\not\in P$. Then by Lemma 3.4, we have
$$\sum_{p\le x,p\in P}\frac{\rho_N(p)}{p}=\sum_{p\le x}
\frac{\rho_N(p)}{p}-\sum_{p\le x,p\not\in P}\frac{\rho_N(p)}
{p}=\log\log x+{\rm constant}+O(1/\log x).$$
Furthermore, since $N$ is the normal closure of $L/\mathbb{Q}$, a prime number $p$ splits completely in $N$ if and only if it splits completely in $L$. That is $\rho(p)=[L:\mathbb{Q}]$ and $\rho_N(p)=[N:\mathbb{Q}]$ for all $p\in P$. Therefore
$$\sum_{p\le x, p\in P}\frac{\rho(p)}{p}
=\frac{1}{[N:L]}\sum_{p\le x,p\in P}\frac{\rho_N(p)}{p}
=\frac{\log\log x}{[N:L]}+{\rm constant}+O(1/\log x),$$
Hence the lemma can be derived from
\begin{align*}
\log\Big(\prod_{p\le x \atop p\in P}\Big(1+\frac{\rho(p)}{p}\Big)\Big)
=&\sum_{p\le x, p\in P}\frac{\rho(p)}{p}+{\rm constant}+O(1/x) \\
=&\frac{\log\log x}{[N:L]}+{\rm constant}+O(1/\log x).
\end{align*}
\end{proof}

Now we give the proof of Theorem 1.2.

{\it Proof of Theorem 1.2.}
Let $\{\mathfrak{a}_i\}_{i=1}^{\infty}$ be an arrangement of elements in $S_{\eta}$ such that $\mathfrak{N}(\mathfrak{a}_i)\le\mathfrak{N}(\mathfrak{a}_{i+1})$ for all $i\ge 1$. Then by Weyl's criterion on uniformly distributed sequences, to show that $\{\alpha(\mathfrak{a}_i)/\mathfrak{N}(\mathfrak{a}_i)\}_{i=1}^{\infty}$ is uniformly distributed, it is equivalent to prove for all $h\neq0$ that
$$\lim_{N\rightarrow\infty}\frac{1}{N}\sum_{i=1}^{N}
e(h\alpha(\mathfrak{a}_i)/\mathfrak{N}(\mathfrak{a}_i))=0. \leqno(3.1)$$
Suppose that $\mathfrak{N}(\mathfrak{a}_N)=x$. Then
$N=\sum_{n\le x}\rho_{\eta}(n)+o(x)$ and
$$\sum_{i=1}^{N}e(h\alpha(\mathfrak{a}_i)/\mathfrak{N}(\mathfrak{a}_i))
=\sum_{n\le x}\rho_{\eta}(h,n)+o(x), h\neq 0.$$
Since $\sum_{n\le x}\rho_{\eta}(n)=cx+o(x)$ by Lemma 3.3, (3.1) is equivalent to
$$\lim_{x\rightarrow\infty}\frac{\sum_{n\le x}\rho_{\eta}(h,n)}
{\sum_{n\le x}\rho_{\eta}(n)}=0,\ \forall h\neq 0.$$
Again by $\sum_{n\le x}\rho_{\eta}(n)\thicksim cx$, we need only to show that $\sum_{n\le x}\rho_{\eta}(h,n)=o(x)$ for nonzero integer $h$.

We actually go farther than this. Instead of considering a nonzero constant $h$, we will estimate the sum $\sum_{n\le x}\rho_{\eta}(h(n),n)$ for a function $h$ defined on the set $\{n>0: (n,\eta)=1\}$ such that: (i) $h(n)$ is always a nonzero integer; (ii) for any $m|n$, $h(n)\equiv h(m)\mod m$; (iii) there exists $C>0$ such that $\gcd(h(n),n)<C$ for all $n$. A particular example that one needs keep in mind is as following. Let $m$ be an positive integer such that all of its prime factors divide $\eta$. For a positive integer $n$ prime to $\eta$, $h(n)$ is defined to be the unique integer in $\{1,...,n\}$ such that $mh(n)\equiv 1\mod n$. Obviously such a function satisfies all the conditions listed.  To simplify the notation, we will still use $h$ other than $h(n)$ to represent a function.

Let $X:=x^{\frac{1}{24el\log\log x}}$, $\xi(n):=\prod_{p<X}p^{v_p(n)}$, where $v_p(n)$ denotes $p$-adic value of $n$. Write
\begin{align*}
\sum_{n\le x}\rho_{\eta}(h,n)=\sum_{n\le x \atop \xi(n)\le x^{1/3}}
\rho_{\eta}(h,n)+\sum_{n\le x, \atop \xi(n)> x^{1/3}}\rho_{\eta}(h,n).
\end{align*}
Let $n_1,n_2$ represent positive integers with $\xi(n_1)=n_1, \xi(n_2)=1$. Let $\bar{n}_1$ be the inverse of $n_1\mod n_2$, and $\bar{n}_2$ be the inverse of $n_2\mod n_1$. Since each $n$ can be uniquely written as a product of $n_1n_2$, it follows from Lemma 3.2 that
\begin{align*}
\sum_{n\le x \atop \xi(n)\le x^{1/3}}|\rho_{\eta}(h,n)|
=&\sum_{n_1n_2\le x \atop n_1\le x^{1/3}}
|\rho_{\eta}(h\bar{n}_2,n_1)\rho_{\eta}(h\bar{n}_1,n_2)| \\
\le&\sum_{n_1\le x^{1/3}}\sum_{n_2\le x/n_1}
\rho_{\eta}(n_2)|\rho_{\eta}(h\bar{n}_2,n_1)| \\
\le&\sum_{n_1\le x^{1/3}}
\Big(\sum_{n_2\le x/n_1}\rho_{\eta}(n_2)^2\Big)^{1/2}
\Big(\sum_{n_2\le x/n_1}|\rho_{\eta}(h\bar{n}_2,n_1)|^2\Big)^{1/2}
\end{align*}

So far we do nothing new but repeating Hooley's treatments on the exponential sums. By applying Hooley's method (see estimate of $\Sigma_2$, $\Sigma_5$ and $\Sigma_6$ in \cite{Ho}), we have
$$\sum_{n\le x \atop \xi(n)>x^{1/3}}|\rho_{\eta}(h,n)|
\le\sum_{n\le x \atop \xi(n)>x^{1/3}}\rho_{\eta}(n)
=O\Big(\frac{x}{\log x}\Big)$$
$$\sum_{n_2\le x/n_1}\rho_{\eta}(n_2)^2
=O\Big(\frac{x\log\log^{l^2} x}{n_1\log x}\Big)$$
\begin{align*}
\sum_{n_2\le x/n_1}|\rho_{\eta}(h\bar{n}_2,n_1)|^2
=\sum_{1\le a\le n_1 \atop (a,n_1)=1}|\rho_{\eta}(ah,n_1)|^2
\sum_{n_2\le x/n_1 \atop \bar{n}_2\equiv a\mod n_1} 1\\
\le O\Big(\frac{x\log\log x}{n_1\varphi(n_1)\log x}\Big)
\sum_{a=1}^{n_1}|\rho_{\eta}(ah,n_1)|^2.
\end{align*}
Now we have $\sum_{n\le x, \xi(n)>x^{1/3}}|\rho_{\eta}(h,n)|=o(x)$ and
$$\sum_{n\le x \atop \xi(n)\le x^{1/3}}|\rho_{\eta}(h,n)|
=\sum_{n_1\le x^{1/3}}O\Big(\frac{x\log\log^{(l^2+1)/2}x}
{\sqrt{\varphi(n_1)}n_1\log x}\Big)
\Big(\sum_{a=1}^{n_1}|\rho_{\eta}(ah,n_1)|^2\Big)^{1/2}. \leqno(3.2)$$
To proceed, an upper bound of $\sum_{a=1}^{n_1}|\rho_{\eta}(ah,n_1)|^2$ is needed. We leave this as a lemma below. Note that Hooley's upper bound in Lemma 1 of \cite{Ho} is not applicable in our case.

\begin{lem}\label{keytot2}
$$\sum_{a=1}^{n}|\rho_{\eta}(ah,n)|^2\le \frac{n(h,n)\rho_{\eta}(n)^2}
{d^{\omega_P(n)}}.$$
\end{lem}
\begin{proof}
\begin{align*}
\sum_{a=1}^{n}|\rho_{\eta}(ah,n)|^2=&\sum_{a=1}^{n}
\sum_{\mathfrak{a}\in S_{\eta}(n)}
e\Big(\frac{ah\alpha(\mathfrak{a})}{n}\Big)
\sum_{\mathfrak{a}\in S_{\eta}(n)}
e\Big(-\frac{ah\alpha(\mathfrak{a})}{n}\Big)\\
=&\sum_{a=1}^{n}\sum_{\mathfrak{a,b}\in S_{\eta}(n)}
e\Big(\frac{ah(\alpha(\mathfrak{a})-\alpha(\mathfrak{b}))}{n}\Big) \\
=&n|\{(\mathfrak{a,b})\in S_{\eta}^2(n):\alpha(\mathfrak{a})\equiv
\alpha(\mathfrak{b})\mod n/(h,n)\}| \\
=&n\sum_{v=0}^{n-1}|\{\mathfrak{a}:\alpha(\mathfrak{a})=v\}|
\sum_{i=1}^{(h,n)}|\{\mathfrak{b}: \alpha(\mathfrak{b})\equiv
v+in/(h,n)\mod n\}| \\
\le&n(h,n)\sum_{v=1}^{n}|\{\mathfrak{a}:\alpha(\mathfrak{a})=v\}|^2
=:n(h,n)g(n),
\end{align*}
where
$$g(n)=\sum_{v=1}^{n}|\{\mathfrak{a}:\alpha(\mathfrak{a})=v\}|^2
=|\{(\mathfrak{a,b})\in S_{\eta}^2(n):\alpha(\mathfrak{a})
=\alpha(\mathfrak{b})\}|.$$
It is easy to check that $g(n)$ is multiplicative, i.e., $g(n_1n_2)=g(n_1)g(n_2)$, provided that $(n_1,n_2)=1$. Meanwhile, $g(p^a)\le g(p)$  for all $a>0$, and the equal happens when $p$ is unramified. So we need only discuss on $g(p)$ and then can get an upper bound for $g(n)$ by multiplicativity.

For $p\not\in P$, we take the trivial bound $g(p)\le \rho_{\eta}(p)^2$. For $p\in P$, there are exactly $d$ different roots for $f(x)\equiv 0\mod p$ and $|\{\mathfrak{a}:\alpha(\mathfrak{a})=v\}|^2=[L:K]^2$
for each root $v$. Thus
$$g(p)=d[L:\mathbb{Q}(\alpha)]^2=\rho_{\eta}(p)^2/d.$$
Therefore
$$g(n)\le \prod_{p|n \atop p\in P}\frac{\rho_{\eta}(p^{v_p(n)})^2}{d}
\prod_{p|n \atop p\not\in P}\rho_{\eta}(p^{v_p(n)})^2
=\frac{\rho_{\eta}(n)^2}{d^{\omega_P(n)}}.$$
This proves Lemma \ref{keytot2}.
\end{proof}

Let's go back to the estimate of
$$\sum_{n\le x, \xi(n)\le x^{1/3}}|\rho_{\eta}(h,n)|.$$
Applying Lemma 3.6 to (3.2), we deduce that
$$\sum_{n\le x \atop \xi(n)\le x^{1/3}}|\rho_{\eta}(h,n)|
=O\Big(\frac{x\log\log^{(l^2+1)/2}x}{\log x}
\sum_{n_1\le x^{1/3}}\frac{\rho_{\eta}(n_1)}{\sqrt{n_1
\varphi(n_1)d^{\omega_P(n_1)}}}\Big).\leqno(3.3)$$
It is left to estimate $\sum_{n_1\le x^{1/3}}\frac{\rho_{\eta}(n_1)}
{\sqrt{n_1\varphi(n_1)d^{\omega_P(n_1)}}}$.

\begin{lem}
$$\sum_{n\le x}\frac{\rho_{\eta}(n)}{\sqrt{n\varphi(n)d^{\omega_P(n)}}}
=O\Big(\log^{1-\frac{1}{\sqrt{d}[N:L]}}x\Big).$$
\end{lem}
\begin{proof}
Since $\frac{\rho_{\eta}(n)}{\sqrt{n\varphi(n)d^{\omega_P(n)}}}$ is a multiplicative arithmetic function, it follows form Lemma 3.5 that
\begin{align*}
\sum_{n\le x}\frac{\rho_{\eta}(n)}{\sqrt{n\varphi(n)d^{\omega_P(n)}}}
\le&\prod_{p\le x}\Big(1+\frac{\rho_{\eta}(p)}{\sqrt{p\varphi(p)
d^{\omega_P(p)}}}+\frac{\rho_{\eta}(p^2)}{\sqrt{p^2\varphi(p^2)
d^{\omega_P(p^2)}}}+\cdots \Big) \\
=&\prod_{p\le x \atop p\not\in P}\Big(1+\frac{\rho_{\eta}(p)}{p}
+O\Big(\frac{1}{p^{2}}\Big)\Big)
\prod_{p\le x \atop p\in P}\Big(1+\frac{\rho_{\eta}(p)}{p\sqrt{d}}
+O\Big(\frac{1}{p^{2}}\Big)\Big) \\
=&O\Big(\prod_{p\le x }\Big(1+\frac{\rho_{\eta}(p)}{p}\Big)
\prod_{p\le x \atop p\in P}\Big(1+\frac{\rho_{\eta}(p)}
{p}\Big)^{1-1/\sqrt{d}}\Big) \\
=&O\Big(\log^{1-\delta}x\Big),
\end{align*}
where $\delta=\frac{1-1/\sqrt{d}}{[N:L]}$.
\end{proof}

By Lemma 3.7 and (3.3), we have
$$\sum_{n<x \atop \xi(n)\le x^{1/3}}|\rho_{\eta}(h,n)|
=O\Big(\frac{x\log\log^{(l^2+1)/2}x}
{\log^{\delta}x}\Big).$$
Hence we conclude that
$$\sum_{n\le x}|\rho_{\eta}(h,n)|=O\Big(\frac{x\log\log^{(l^2+1)/2}x}
{\log^{\delta}x}\Big)=o(x). \leqno(3.4)$$
This proves Theorem 1.2. \hfill$\Box$ \\

As we can see, we largely follow Hooley's idea to estimate the exponential sum $\sum_{n\le x}\rho_{\eta}(h,n)$. The upper bound obtained by this method is far away from good. An expected bound is at least a smaller than one power of $x$. But one also should notice that this method is powerful if we merely consider uniform distribution. It is also worth mention that we actually get $\sum_{n\le x}|\rho_{\eta}(h,n)|=o(x)$. It allows us to get a generalization of Theorem 1.2.
\begin{thm}
Let $A$ be any set of positive integers and $S_{\eta,A}:=\cup_{n\in A}S_{\eta}(n)$. If there is a constant $c>0$ such that $\sum_{n\le x, n\in A}\rho_{\eta}(n)>cx$, then the sequence
$$\{\alpha(\mathfrak{a})/\mathfrak{N}(\mathfrak{a})
\}_{\mathfrak{a}\in S_{\eta,A}}$$
is uniformly distributed.
\end{thm}

\begin{proof}
By the similar discussion as in the beginning of the proof of Theorem 1.2, the sequence $\{\alpha(\mathfrak{a})/\mathfrak{N}(\mathfrak{a})
\}_{\mathfrak{a}\in S_{\eta,A}}$ is uniformly distributed if and only if
$$\lim_{x\rightarrow\infty}\frac{\sum_{n\le x,n\in A}\rho_{\eta}(h,n)}
{\sum_{n\le x,n\in A}\rho_{\eta}(n)}=0,\ \forall h\neq0.$$
Since $\sum_{n\le x,n\in A}\rho_{\eta}(n)>cx$, it is enough to show
$\sum_{n\le x,n\in A}\rho_{\eta}(h,n)=o(x)$. This is obviously true since by (3.4),
$$\sum_{n\le x,n\in A}\rho_{\eta}(h,n)\le \sum_{n\le x}
|\rho_{\eta}(h,n)|=o(x).$$
This end the proof of Theorem 3.8.
\end{proof}

It is worth mention that when $A$ is the set of squarefree positive integers, Theorem 3.8 is an immediate
corollary of a much more general and intrinsic work of Kowalski and Soundararajan \cite{KS}.

\section{roots of a system of polynomial congruences}

In this section, we will apply Theorem 1.2 to study the distribution of the roots of a system of polynomial congruences. First recall the Weyl's criterion on uniformly distributed sequences of higher dimension.

\begin{lem}
Let $a_n:=(a_n^{(1)},...,a_n^{(r)}), n>0$ be a sequence of elements in $[0,1]^r$. Then $\{a_n\}_{n>0}$ is uniformly distributed if and only if
$$\sum_{n\le x}e(h_1a_n^{(1)}+\cdots+h_ra_n^{(r)})=o(x)$$
for any $(h_1,...,h_r)\in\mathbb{Z}^r\setminus\{(0,...,0)\}$.
\end{lem}

The Theorem 3.8 can be extended to higher dimension.

\begin{thm}
Let $\alpha_1,...,\alpha_r$ be algebraic numbers in $L$ such that $1,\alpha_1,...,\alpha_r$ are linearly independent over $\mathbb{Q}$ and $\eta$ be an positive integer such that $\eta\alpha_i$ are algebraic integers for all $1\le i\le r$. Let $A$ and $S_{\eta,A}$ be as in Theorem 3.8. Then the sequence
$$\Big\{\Big(\frac{\alpha_1(\mathfrak{a})}{\mathfrak{N}(\mathfrak{a})},
...,\frac{\alpha_r(\mathfrak{a})}{\mathfrak{N}(\mathfrak{a})}
\Big)\Big\}_{\mathfrak{a}\in S_{\eta,A}} \leqno(4.1)$$
is uniformly distributed.
\end{thm}

\begin{proof}
For $(h_1,...,h_r)\in\mathbb{Z}^r\setminus\{0\}$, let $\alpha=h_1\alpha_1+\cdots+h_r\alpha_r$. Then for each degree one ideal $\mathfrak{a}\subset\mathcal{O}[\frac{1}{\eta}]$, $$\alpha(\mathfrak{a})\equiv h_1\alpha_1(\mathfrak{a})+\cdots
+h_r\alpha_r(\mathfrak{a})\mod \mathfrak{N}(\mathfrak{a}).$$
Since $1,\alpha_1...,\alpha_r$ are linearly independent over $\mathbb{Q}$, $\alpha$ is irrational if one of $h_1,...,h_r$ is nonzero. It then follows from (3.4) that
$$\sum_{n\le x}\sum_{\mathfrak{a}\in S_{\eta,A}(n)}
e\Big(\frac{h_1\alpha_1(\mathfrak{a})+\cdots+
h_r\alpha_r(\mathfrak{a})}{\mathfrak{N}(\mathfrak{a})}\Big)
=\sum_{n\le x}\sum_{\mathfrak{a}\in S_{\eta,A}(n)}
e\Big(\frac{\alpha(\mathfrak{a})}{n}\Big)=o(x).$$
Hence by high dimensional Weyl's criterion, the sequence (4.1) is uniformly distributed.
\end{proof}
Let $f_1,...,f_r$ and $D_1,...,D_r$ be given as in Theorem 1.3. For each $i$ with $1\le i\le r$, let $\alpha_i$ be a root of $f_i$ and $K_i=\mathbb{Q}(\alpha_i)$.  Write $L=\mathbb{Q}(\alpha_1,...,\alpha_r)$. Since $(D_i,D_j)=1$ for all $1\le i<j\le r$, it follows from Lemma 2.2 that
$$[L:\mathbb{Q}]=[K_1:\mathbb{Q}]\cdots[K_r:\mathbb{Q}].$$
Write $D=\prod_{i=1}^r D_r$.
The following Theorem 4.3 extend the furthermore part of Theorem 1.1.

\begin{thm}
Let $\eta$ be an positive integer such that $\eta\alpha_i$ are algebraic integers. For $(n, \eta D)=1$, there is an bijection between $S_{\eta}(n)$ and $\{(v_1,...,v_r): f_i(v_i)\equiv 0\mod n, 1\le i\le r\}$,
given by
$$\mathfrak{a}\mapsto(\alpha_1(\mathfrak{a}),...,
\alpha_r(\mathfrak{a}))=(v_1,...,v_r)$$
and
$$(v_1,...,v_r)\mapsto(\alpha_1-v_1,...,\alpha_r-v_r,n)=\mathfrak{a}.$$
\end{thm}
\begin{proof}
We first show that
$$\mathfrak{p}:=(\alpha_1-v_1,...,\alpha_r-v_r,p)
\subset\mathcal{O}[1/\eta]$$
is a degree one prime ideal over $p$ by induction on $r$, where $\mathcal{O}$ denote the ring of integers of $L$. When $r=1$, this is true by Dedekind's Theorem. Suppose $r>1$ and $(\alpha_1-v_1,...,\alpha_i-v_i,p)$ is a degree one ideal in $\mathcal{O}[\frac{1}{\eta}]\cap\mathbb{Q}(\alpha_1,...,\alpha_i)$ for $1\le i\le r-1$. Denote by $\mathcal{O}'$ the ring of integers of $\mathbb{Q}(\alpha_1,...,\alpha_{r-1})$ and $\mathfrak{p}':=(\alpha_1-v_1,...,\alpha_{r-1}-v_{r-1},p)$. Since
$$[L:\mathbb{Q}(\alpha_1,...,\alpha_{r-1})]=[K_r:\mathbb{Q}]=\deg f_r,$$
the minimal polynomial $f_r(x)$ of $\alpha_r$ over $\mathbb{Q}$ is again its minimal polynomial over $\mathbb{Q}(\alpha_1,...,\alpha_{r-1})$.
By induction assumption, $\mathfrak{p}'$ is of degree one. So
$$\mathbb{Z}[1/\eta]/p\mathbb{Z}[1/\eta]
\cong\mathcal{O}'[1/\eta]/\mathfrak{p}'.$$
Hence a root of $f_r(x)\equiv 0\mod p$ is a root of $f_r(x)\equiv 0\mod \mathfrak{p}'$. Now applying Dedekind's factorization theorem to $\mathfrak{p}'$ and the extension $\mathcal{O}[\frac{1}{\eta}]/\mathcal{O}'[\frac{1}{\eta}]$, one gets
$(\alpha_r-v_r, \mathfrak{p}')$, i.e., $(\alpha_1-v_1,...,\alpha_r-v_r,p)$, is a prime ideal of degree one in $\mathcal{O}[\frac{1}{\eta}]$.

We then show that $(\alpha_1-v_1,...,\alpha_r-v_r,p^e)=\mathfrak{p}^e$, provided $e=v_p(n)$. This is also proved by induction on $r$. When $r=1$, we have showed that $(\alpha_1-v_1,p^e)=(\alpha_1-v_1,p)^e$ in the proof of Theorem 1.1. Suppose that
$$(\alpha_1-v_1,...,\alpha_i-v_i,p^e)=(\alpha_1-v_1,...,\alpha_i-v_i,p)^e$$ for $1\le i\le r-1$. Then $(\alpha_1-v_1,...,\alpha_r-v_r,p^e)=(\alpha_r-v_r,\mathfrak{p}'^e)$.
It is then left to prove $(\alpha_r-v_r,\mathfrak{p}'^e)=(\alpha_r-v_r,\mathfrak{p}')^e$. This can be done by the same argument of proving the case $r=1$.

Finally we show
$$(\alpha_1-v_1,...,\alpha_r-v_r,n)=\prod_{p|n}
(\alpha_1-v_1,...,\alpha_r-v_r,p^{v_p(n)}),$$
from which the theorem follows. On the one hand,
$$(\alpha_1-v_1,...,\alpha_r-v_r,n)\subset
(\alpha_1-v_1,...,\alpha_r-v_r,p^{v_p(n)})$$
for each $p|n$. Thus
$$(\alpha_1-v_1,...,\alpha_r-v_r,n)\subset\prod_{p|n}
(\alpha_1-v_1,...,\alpha_r-v_r,p^{v_p(n)}).$$
On the other hand, by comparing the generators, we get
$$\prod_{p|n}(\alpha_1-v_1,...,\alpha_r-v_r,p^{v_p(n)})
\subset(\alpha_1-v_1,...,\alpha_r-v_r,n).$$
This implies the desired equation and proves Theorem 4.3.
\end{proof}

To simplify the notations, we let $\mathbf{h,v}$ denote r-tuple of integers. Suppose that all components of $\mathbf{v}$ are nonnegative. For an integer $a$, define $a\mathbf{h}=(ah_1,...,ah_r)$ and $\mathbf{h}\cdot\mathbf{v}=h_1v_1+\cdots +h_rv_r$.  Let $\mathbf{f}(x)=(f_1(x),...,f_r(x))$,
$\mathbf{f}(\mathbf{v})=(f_1(v_1),...,f_r(v_r))$. We say $\mathbf{f}(\mathbf{v})\equiv 0\mod n$ if $f_i(v_i)\equiv 0\mod n$ for $1\le i\le r$. Denote
$$\rho(\mathbf{h},n)=\sum_{\mathbf{f}(\mathbf{v})\equiv 0 (n)}e(\mathbf{h}\cdot\mathbf{v}/n).$$
Let $\mathbf{0}=(0,...,0)$ and $d=\max_{1\le i\le r}\deg f_i$. Then
$$|\rho(\mathbf{h},n)|\le\rho(\mathbf{0},n)
=\prod_{i=1}^r|\{v_i: f_i(v_i)\equiv 0\mod n\}|
=O(d^{r\omega(n)}).$$

\begin{lem}
Let $n,n_i,\bar{n}_i,i=1,2$ be given as in Lemma 3.2. Then
$$\rho(\mathbf{h},n)=\rho(\bar{n}_2\mathbf{h},n_1)
\rho(\bar{n}_1\mathbf{h},n_2).$$
\end{lem}
\begin{proof}
Let $\mathbf{v}_i$ be solution of $\mathbf{f}(x)\equiv 0\mod n_i$ for $i=1,2$. By Chinese Remainder Theorem, each root $\mathbf{v}$ of $\mathbf{f}(x)\equiv 0\mod n$ can be uniquely write as
$$\mathbf{v}=n_2\bar{n}_2\mathbf{v}_1+n_1\bar{n}_1\mathbf{v}_2 \mod n.$$
So
\begin{align*}
\rho(\mathbf{h},n)
=&\sum_{\mathbf{f}(\mathbf{v})\equiv 0 (n)}
e(\mathbf{h}\cdot\mathbf{v}/n) \\
=&\sum_{\mathbf{f}(\mathbf{v}_i)\equiv 0 (n_i)}
e(\mathbf{h}\cdot(n_2\bar{n}_2\mathbf{v}_1
+n_1\bar{n}_1\mathbf{v}_2)/n) \\
=&\rho(\bar{n}_2\mathbf{h},n_1)
\rho(\bar{n}_1\mathbf{h},n_2)
\end{align*}
\end{proof}

We now prove Theorem 1.3.

{\it Proof of Theorem 1.3}.
For $\mathbf{h}\neq\mathbf{0}$, by Lemma 4.4,
\begin{align*}
\sum_{n\le x}|\rho(\mathbf{h},n)|
=&\sum_{n_1n_2\le x, \atop {(n_2,\eta D)=1
\atop p|n_1\Rightarrow p|\eta D}}
|\rho(\bar{n}_2\mathbf{h},n_1)\rho(\bar{n}_1\mathbf{h},n_2)| \tag{4.2} \\
\le &\sum_{n_1\le x \atop p|n_1\Rightarrow p|\eta D}
\rho(\mathbf{0},n_1)
\sum_{n_2\le x/n_1 \atop (n_2,\eta D)=1}
|\rho(\bar{n}_1\mathbf{h},n_2)|
\end{align*}
where $\bar{n}_1, \bar{n}_2$ satisfy $n_1\bar{n}_1 \equiv 1\mod n_2$ and
$n_2\bar{n}_2 \equiv 1\mod n_1$.
Let $\alpha=h_1\alpha_1+\cdots+h_r\alpha_r$. By Theorem 4.3,
$$\sum_{n_2\le x/n_1 \atop (n_2,\eta D)=1}
|\rho(\bar{n}_1\mathbf{h},n_2)|
=\sum_{n_2\le x/n_1 \atop (n_2,\eta D)=1}
\sum_{\mathfrak{a}\in S_{\eta}(n_2)}e\Big(
\frac{\bar{n}_1\alpha(\mathfrak{a})}{n_2}\Big).$$
Note that $\bar{n}_1$ is a function on $n_2$ satisfying the conditions listed in the proof of Theorem 1.2. So by (3.4),
$$\sum_{n_2\le x/n_1 \atop (n_2,\eta D)=1}
|\rho(\bar{n}_1\mathbf{h},n_2)|=o(x/n_1).$$
For a positive integer $n_1$ with all of its prime factors dividing $\eta D$, one has $\rho(\mathbf{0},n_1)\le O(d^{r\omega(\eta D)})$. Thus from (4.2) we derive that
$$\sum_{n\le x}|\rho(\mathbf{h},n)|\le C\sum_{n_1\le x \atop
p|n_1\Rightarrow p|\eta D}o(x/n_1)=o(x)$$
for some positive constant $C$. Therefore by Lemma 4.1, the sequence of ratios $(v_1/n,...,v_r/n)$ is uniformly distributed. \hfill$\Box$

\section{distribution of digits of $n$-adic expansions}
In this section we study the distribution of sequences concerning the digits of $n$-adic expansions of irrational algebraic numbers. Let $\alpha$ be an irrational algebraic number and $f(x)$ be the primitive minimal polynomial of $\alpha$ over $\mathbb{Z}$. Suppose that $\alpha$ has an $n$-adic expansion. Since $\mathbb{Z}_n\cong\prod_{p|n}\mathbb{Z}_p$, it implies that $f(x)$ has a solution in $\mathbb{Z}_p$ for $p|n$. So $\mathbb{Q}(\alpha)$ is a subfield of $\mathbb{Q}_p$. Let $\mathcal{O}$ be the ring of algebraic integers in $\mathbb{Q}(\alpha)$. Then $p\mathbb{Z}_p\cap\mathcal{O}$ is a unramified prime ideal of degree one of $\mathcal{O}$. The product $\prod_{p|n}(p\mathbb{Z}_p\cap\mathcal{O})^{v_p(n)}$ is an degree one unramified ideal with norm $n$.

Conversely, let $\mathfrak{a}$ be an unramified ideal of degree one with norm $n$. Write the principal fractional ideal $(\alpha)$ uniquely as a reduced ratio of integral ideals $\mathfrak{c}/\mathfrak{b}$. If $\mathfrak{a}$ is prime to $\mathfrak{b}$, then we have
$$\alpha\in\varprojlim_{l} \mathcal{O}/\mathfrak{a}^l
\cong\varprojlim_{l}\mathbb{Z}/n^l\mathbb{Z}.$$
So $\alpha$ has an $n$-adic expansion. Thus by above discussion we set up an one to one correspondence between the $n$-adic expansions of $\alpha$ and the degree one ideal with norm $n$ which is prime to the denominator of $(\alpha)$.

Now let $\alpha=\sum_{l=0}^{\infty}a_ln^l$ be an $n$-adic expansion of $\alpha$, and $\mathfrak{a}$ be the corresponding ideal. Choose any $\beta\in\mathfrak{b}\setminus\mathfrak{a}$ and denote by $\gamma:=\beta\alpha$. For any $l\ge 1$, define $\alpha(\mathfrak{a}^l)$ to be the unique integer in $\{0,1,...,n^l-1\}$ such that
$$\gamma-\beta \alpha(\mathfrak{a}^l)\in\mathfrak{a}^l.$$
Evidently $\alpha(\mathfrak{a}^l)$ is independent on the choice of $\beta$. This definition of $\alpha(\mathfrak{a}^l)$ coincides with the one we gave in the introduction if $\mathfrak{a}$ is prime to an integer $\eta$. Under this definition we have for all $l\ge 1$ that
$$\alpha(\mathfrak{a}^l)=a_0+a_1n+\cdots+a_{l-1}n^{l-1}. \leqno(5.1)$$
It implies that
$$\frac{a_{l-1}}{n}=\frac{\alpha(\mathfrak{a}^l)}{n^l}
+O\Big(\frac{1}{n}\Big). $$
So for fixed $l>0$, the sequence of ratios $a_{l-1}/n$ is uniformly distributed if and only if $\alpha(\mathfrak{a}^l)/n^l$ does.

Be aware that here $\mathfrak{a}$ runs through all unramified degree one ideals that are prime to the ideal $\mathfrak{b}$. It means that $\mathfrak{a}$ does not run through any set $S_{\eta, A}$ given in Theorem 3.8. So for $l=1$, we can not get the uniformity of the sequence of ratios $a_{l-1}/n$ directly from Theorem 3.8. But we do know that some of its subsequence are uniformly distributed by Theorem 1.2. In the following, we show that we can get the uniformity of the sequence of ratios $a_0/n$ by the uniformity of these subsequences.

Let $\eta$ be the least positive integer such that $\eta\alpha$ is integral, and $\mathfrak{b}$ be the denominator of $(\alpha)$ given above.
Denote by $D$ the discriminant of the primitive minimal polynomial of $\alpha$ over $\mathbb{Z}$. Theorem 1.2 gives us the uniformity of the sequence of ratios $\alpha(\mathfrak{a})/\mathfrak{N}(\mathfrak{a})$ as $\mathfrak{a}$ runs through all the the degree one ideals prime to $\eta D$.

Now consider any unramified degree one ideal $\mathfrak{a}$ that is prime to $\mathfrak{b}$. Suppose the norm of $\mathfrak{a}$ is $n$. Write $n=n_1n_2$ such that $p|n_1\Rightarrow p|\eta D$ and $(n_2,\eta D)=1$. Then there exists a unique pair of unramified degree one ideals $\mathfrak{a}_1$ with norm $n_1$ and $\mathfrak{a}_2$ with norm $n_2$ such that $\mathfrak{a}=\mathfrak{a}_1\mathfrak{a}_2$. We have $\alpha(\mathfrak{a}^l)-\alpha(\mathfrak{a}_1^l)\equiv 0\mod n_1^l$ and hence
$$\frac{\alpha-\alpha(\mathfrak{a}_1^l)}{n_1}\equiv \frac{\alpha(\mathfrak{a}^l)-\alpha(\mathfrak{a}_1^l)}{n_1^l}
\mod \mathfrak{a}_2^l.$$
In other words, $(\alpha(\mathfrak{a}^l)-\alpha(\mathfrak{a}_1^l))/n_1^l$ is the residue of $(\alpha-\alpha(\mathfrak{a}_1))/n_1^l$ modulo $\mathfrak{a}_2^l$. Denote by $U(n)$ the set of unramified degree one ideals with norm $n$. For any fixed $\mathfrak{a}_1$, denote by $\beta:=(\alpha-\alpha(\mathfrak{a}_1))/n_1^l$. If
$$\sum_{n_2<x/n_1 \atop (n_2,\eta D)=1}\sum_{\mathfrak{a}_2\in U(n_2)}
e(h\beta(\mathfrak{a}_2^l)/n_2^l)=o(x/n_1)$$
for all $n_1$ with $p|n_1\Rightarrow p|\eta D$ and $\mathfrak{a}_1\in U(n_1)$, then
\begin{align*}
&\sum_{n\le x}\sum_{\mathfrak{a}\in U(n)}e(h\alpha(\mathfrak{a}^l)/n^l) \\
=&\sum_{n_1\le x \atop p|n_1\Rightarrow p|\eta D}
\sum_{\mathfrak{a}_1\in U(n_1)}
\sum_{n_2\le x/n_1  \atop (n_2,\eta D)=1}
\sum_{\mathfrak{a}_2\in U(n_2)}
e\Big(h\cdot\Big(\frac{((\alpha-\alpha(\mathfrak{a}_1^l))/n_1^l)
(\mathfrak{a}_2)}{n_2}+O({1}/{n_2^l})\Big) \\
=&\sum_{n_1\le x \atop p|n_1\Rightarrow p|\eta D}
\sum_{\mathfrak{a}_1\in U(n_1)}o(x/n_1)
=o(x).
\end{align*}
Therefore if the sequence (1.1) with $S'$ given by (b) is uniformly distributed for all irrational algebraic numbers $\alpha$, then the sequence of ratios $a_{l+1}/n$ is also uniformly distributed. Specially for $l=0$, the uniformity of the sequence of ratios $a_0/n$ follows from Theorem 1.2.

It is interesting to compare this method with the one used in the proof of Theorem 1.3. Though it would be more complicated, this method can also be used to prove Theorem 1.3. The way to prove Theorem 1.3 is also applicable here. But the method in the proof of Theorem 1.3 uses more than just uniformity of subsequences. \\

In the following we study $n$-adic normal numbers. First we explore the relationship between normal numbers and uniformly distributed sequences.

\begin{lem}
Let $\alpha=\sum_{l=0}^{\infty}a_ln^l$ be an $n$-adic number. Denote by $\alpha_l=a_0+\cdots+a_{l-1}n^{l-1}$. Then $\alpha$ is normal if and only is $\{\alpha_l/n^l\}_{l=1}^{\infty}$ is uniformly distributed.
\end{lem}
\begin{proof}
If $\{\alpha_l/n^l\}_{l=1}^{\infty}$ is uniformly distributed, then for any positive integer $m$,
$$\lim_{x\rightarrow\infty}\frac{|\{l<x:\alpha_l/n^l
\in[i/n^m,(i+1)/n^m]\}|}{x}=\frac{1}{n^m}.$$
Thus for $0\le i\le n^m-1$, one has
$$|\{l<x: a_{l-m}+\cdots+a_{l-1}n^{m-1}=i\}|=O(x/n^m)+o(x).$$
So $\alpha$ is normal.

Suppose $\sum_{l=0}^{\infty}a_ln^l$ is normal. For $m>0$, denote by $\Sigma_m:=\{0,...,n-1\}^m$. For $w=(a_0,...,a_{m-1})\in\Sigma_m$, let $w(n):=a_0+a_1n+\cdots a_{m-1}n^{m-1}$. Then for given $m$ and $w\in\Sigma_m$,
$$|\{l<x: (a_l,...,a_{l+m-1})=w\}|=x/n^m+o(x).$$
Equivalently,
$$|\{l<x: (a_{l+m-1},a_{l+m-2},...,a_l)=w\}|=x/n^m+o(x).$$
Then for all positive integers $m$,
\begin{align*}
\sum_{l<x}e\Big(\frac{\alpha_l}{n^l}\Big)
=&\sum_{l<x}e\Big(\frac{a_0+a_1n+\cdots a_{l-1}n^{l-1}}{n^l}\Big)\\
=&\sum_{w\in\Sigma_m}\sum_{m<l<x \atop (a_{l-m},...,a_l)=w}
e\Big(\frac{a_0+a_1n+\cdots a_{l-1}n^{l-1}}{n^l}\Big)+O(1)\\
=&\Big(\frac{x}{n^m}+o(x)\Big)\sum_{w\in\Sigma_m}
e\Big(\frac{w(n)}{n^m}+O\Big(\frac{1}{n^m}\Big)\Big)+O(1) \\
=&O\Big(\frac{x}{n^m}\Big)
\end{align*}
This implies $\sum_{l<x}e(\alpha_l/n^l)=o(x)$. That is, the sequence  $\{\alpha_l/n^l\}_{l=1}^{\infty}$ is uniformly distributed.
\end{proof}
Following Lemma 5.1 and equation (5.1), we have the following equivalence for irrational algebraic numbers $\alpha$.
\begin{cor}
Let $\alpha$ be an irrational algebraic number and $\sum_{l=1}^{\infty}a_ln^l$ be an $n$-adic expansion of $\alpha$. Let $\mathfrak{a}$ be the integral ideal of $\mathbb{Q}(\alpha)$ corresponding to this $n$-adic expansion. Then the $\sum_{l=1}^{\infty}a_ln^l$ is normal if and only if $\{\alpha(\mathfrak{a}^l)/n^l\}_{l=1}^{\infty}$ is uniformly distributed.
\end{cor}
We can restate Conjecture 1.4 by the distribution of residues of $\alpha$ modulo $\mathfrak{a}^l$.
\begin{conj}
Let $L$ be an number field and $\alpha\in L$ be irrational. Let $\mathfrak{a}$ be an unramified ideal of degree one that is prime to the dominator of the principal ideal $(\alpha)$. Then the sequence $\{\alpha(\mathfrak{a}^l)/n^l\}_{l=1}^{\infty}$ is uniformly distributed.
\end{conj}

Finally we give the following fact to support the normal number conjecture. The ideal of the proof is originally from \cite{KN}.
\begin{thm}
Almost all $n$-adic numbers are normal with respect to the Haar measure.
\end{thm}

\begin{proof}
Let $x=\sum_{l=0}^{\infty}a_ln^l$ and $x_l=a_0+a_1n+\cdots+a_{l-1}n^{l-1}$ for all $l\ge 0$. Define
$$S(N,x):=\frac{1}{N}\sum_{l=0}^{N}e\Big(\frac{hx_l}{n^l}\Big).$$
Let $\mu$ denote the Haar measure on $\mathbb{Z}_n$. Then
$$\int_{\mathbb{Z}_n}|S(N,x)|^2d\mu=\frac{1}{N^2}\sum_{k,l=1}^{N}
\int_{\mathbb{Z}_n}e\Big(\frac{hx_l}{n^l}-\frac{hx_k}{n^k}\Big)d\mu.$$
Analysis the integral
$\int_{\mathbb{Z}_n}e\Big(\frac{hx_l}{n^l}-\frac{hx_k}{n^k}\Big)d\mu$.
Without loss of generality, suppose $l\ge k$. Then $\frac{hx_l}{n^l}-\frac{hx_k}{n^k}$ is constant for all $x\in i+n^l\mathbb{Z}_n$ for each integer $i$ with $0\le i\le n^l-1$. So
\begin{align*}
\int_{\mathbb{Z}_n}e\Big(\frac{hx_l}{n^l}-\frac{hx_k}{n^k}\Big)d\mu
=&\sum_{i=0}^{n^l-1}\int_{i+n^l\mathbb{Z}_n}e\Big(\frac{hx_l}{n^l}
-\frac{hx_k}{n^k}\Big)d\mu \\
=&\frac{1}{n^l}\sum_{x=0}^{n^l-1}
e\Big(\frac{hx_l-hn^{l-k}x_k}{n^l}\Big).
\end{align*}
If $l\neq k$, as $x$ running through $0$ to $n^l-1$, the $hx_l-hn^{l-k}x_k$ exactly runs through a complete residue system modulo $n^l$. So
$$\int_{\mathbb{Z}_n}e\Big(\frac{hx_l}{n^l}-\frac{hx_k}{n^k}\Big)d\mu
=\left\{
        \begin{array}{ll}
          0 & \hbox{if $k\neq l$} \\
          1 & \hbox{if $k=l$.}
        \end{array}
\right. $$
Thus
$$\int_{\mathbb{Z}_n}|S(N,x)|^2d\mu=\frac{1}{N}.$$
This implies
$$\sum_{N=1}^{\infty}\int_{\mathbb{Z}_n}|S(N^2,x)|^2d\mu<\infty.$$
By the monotonic convergence theorem,
$$\int_{\mathbb{Z}_n}\sum_{N=1}^{\infty}|S(N^2,x)|^2d\mu
=\sum_{N=1}^{\infty}\int_{\mathbb{Z}_n}|S(N^2,x)|^2d\mu
<\infty.$$
So $\sum_{N=1}^{\infty}|S(N^2,x)|^2<\infty$ for all $x$ but a set of measure zero respect to $\mu$. Hence $\lim_{N\rightarrow\infty}|S(N^2,x)|=0$ for almost all $x$. Now for any positive integer $N$, let $m^2\le N<(m+1)^2$. Then
$$|S(N^2,x)|\le|S(m^2,x)|+\frac{2m}{N}\le|S(m^2,x)|+\frac{1}{\sqrt{N}}.$$
It holds that $\lim_{N\rightarrow\infty}|S(N,x)|=0$ for almost all $x$. The exceptional $x$ contained in a set has measure zero depending on $h$.
Since the union of countably many measure zero set has measure zero. So the Weyl criterion shows that almost all $x$ is normal with respect to the Haar measure.
\end{proof}

\section*{}
{\bf Acknowledgements} Research of the author is supported by NSFC 11901415. This paper is partially done during visiting the School of Mathematics of the University of Minnesota. The author would like to thank the Chinese Scholarship Council (CSC201808515110) for their financial support and Professor Steven Sperber for his hospitality. He also would like to thank Professor Shaofang Hong and Dr. Guoyou Qian for their useful suggestions and discussions.


\bibliographystyle{plain}

\end{document}